\documentclass[12pt]{article}

\usepackage{amsmath,amsthm,amsfonts,amssymb,latexsym,amscd}
\usepackage{color}

\begin{document}

\newtheorem{theorem}{Theorem}
\newtheorem{corollary}[theorem]{Corollary}
\newtheorem{lemma}[theorem]{Lemma}
\newtheorem{proposition}[theorem]{Proposition}
\newtheorem{conjecture}[theorem]{Conjecture}
\newtheorem{definition}[theorem]{Definition}
\newtheorem{problem}[theorem]{Problem}
\newtheorem{remark}[theorem]{Remark}
\newtheorem{example}[theorem]{Example}

\def\A{{\mathcal A}}
\def\B{{\mathcal B}}
\def\C{{\mathcal C}}
\def\D{{\mathcal D}}
\def\F{{\mathcal F}}
\def\H{{\mathcal H}}
\def\J{{\mathcal J}}
\def\K{{\mathcal K}}
\def\N{{\mathcal N}}
\def\OO{{\mathcal O}}
\def\P{{\mathcal P}}
\def\SS{{\mathcal S}}
\def\U{{\mathcal U}}
\def\W{{\mathcal W}}

\def\bC{{\mathbb C}}
\def\bN{{\mathbb N}}
\def\bT{{\mathbb T}}
\def\bZ{{\mathbb Z}}

\def\auto{{\operatorname{Aut}}}
\def\endo{{\operatorname{End}}}
\def\outo{{\operatorname{Out}}}
\def\sp{{\operatorname{span}}}
\def\clsp{{\overline{\operatorname{span}}}}
\def\Ad{\operatorname{Ad}}
\def\id{\operatorname{id}}
\def\di{\operatorname{diag}}

\def\g{{\mathfrak R}{\mathfrak G}_E}

\def\p2{\pi\oplus\pi}

\title{On Normalizers of $C^*$-Subalgebras in the Cuntz Algebra ${\mathcal O}_n$. II}

\author{Tomohiro Hayashi\footnote{Tomohiro Hayashi was supported by JSPS KAKENHI Grant 
Number 25400109.} and  
Wojciech Szyma{\'n}ski\footnote{This work was supported by the FNU Project Grant 
`Operator algebras, dynamical systems and quantum information theory' (2013--2015), 
and the Villum Fonden project grant `Local and global structures of groups and their algebras' 
(2014--2018). }}

\date{\small 9 January 2015}
\maketitle

\renewcommand{\sectionmark}[1]{}

\vspace{7mm}
\begin{abstract}
We investigate subalgebras $A$ of the Cuntz algebra $\OO_n$ that arise as finite direct sums of corners 
of the UHF-subalgebra $\F_n$. For such an $A$, we completely determine its normalizer group inside 
$\OO_n$. 
\end{abstract}

\vfill

\vspace{3mm}

\newpage


\noindent
{\bf Introduction.} 
This note is a continuation of the investigations of $C^*$-subalgebras of the Cuntz algebra 
$\OO_n$ carried out by the first named author in \cite{H}. The main results therein pertain $C^*$-subalgebras 
$A$ of the core UHF-subalgebra $\F_n$ of $\OO_n$ with finite-dimensional relative commutant $A'\cap\F_n$. 
In particular, \cite[Theorem 1.2]{H} says that for such an $A$, the index of the subgroup 
$\{\Ad u|_A : u\in\N_{\F_n}(A)\}$ in $\{\Ad W|_A : W\in\N_{\OO_n}(A)\}$ is finite. The main purpose 
of the present note is to completely determine the structure of normalizer $\N_{\OO_n}(A)$ in the case 
\begin{equation}\label{formofA}
A = \bigoplus_{j=1}^k e_j\F_n e_j,
\end{equation}
where $e_1,\ldots,e_k$ are projections in $\F_n$ such that $\sum_{j=1}^k e_j=1$. The interesting 
and non-trivial aspects of our analysis stem from the fact, observed already in \cite[Example 1.18]{H}, 
that $\N_{\OO_n}(A)$ is not contained in $\F_n$ in general. 

In addition to its intrinsic interest, our work is motivated by its close relation to index theory in the context 
of endomorphisms of the Cuntz algebras, e.g. see  \cite{W,I,CP}. In a more recent paper on this subject, 
\cite{CRS}, endomorphisms of $\OO_n$ globally preserving $\F_n$ are investigated, and we hope that 
the results of the present paper may help shed light on some of the outstanding questions raised therein. 


\vspace{3mm}\noindent
{\bf Notation.} 
For an integer $n\geq 2$, $\OO_n$ is the $C^*$-algebra generated by isometries $S_1,\ldots,S_n$ such that 
$\sum_{i=1}^n S_iS_i^*=1$, \cite{Cun1}. If $\mu=\mu_1\mu_2\cdots\mu_k$ is a word on alphabet 
$\{1,\ldots,n\}$ then we denote $S_\mu=S_{\mu_1}S_{\mu_2}\cdots S_{\mu_k}$, an isometry in $\OO_n$. 
The range projections $P_\mu:=S_\mu S_\mu^*$ corresponding to all words generate a MASA $\D_n$. 
For a word $\mu=\mu_1\cdots\mu_k$ we denote by $|\mu|=k$ its length. Also, we use symbol $\prec$ to 
denote the lexicographic order on words. 

The circle group $U(1)$ acts on $\OO_n$ by $\gamma_z(S_i)
=zS_i$. The fixed-point algebra for this action, denoted $\F_n$, is a UHF-algebra of type $n^\infty$. 
We denote by $\tau:\F_n\to\bC$ the unique normalized trace on $\F_n$. We also let $\varphi:\OO_n\to\OO_n$ 
to be the canonical shift endomorphism, that is 
\begin{equation}\label{shift}
\varphi(x) = \sum_{i=1}^n S_i x S_i^*. 
\end{equation}
For each $x\in\OO_n$ and each generator $S_i$ we have $S_i x=\varphi(x)S_i$. 

If $B$ is a unital $C^*$-algebra then $\U(B)$ denotes the group of its unitary elements. If $A$ is a 
$C^*$-subalgebra of $B$ then $\N_{B}(A):=\{u\in\U(B) : uAu^*=A\}$ is the normalizer of $A$ in $B$. 


\vspace{3mm}\noindent
{\bf The main results.} 
The following lemma and its proof are motivated by Examples 1.17 and 1.18 of \cite{H}. It constitutes 
a technical basis for our further considerations. 

\begin{lemma}\label{lem1}
Let $e,f$ be non-zero projections in $\F_n$, and let $U\in\U(\OO_n)$ be such that $Ue\F_n eU^* = f\F_n f$. 
Then there exists an integer $m$ such that 
$$ \frac{\tau(f)}{\tau(e)} = n^m. $$ 
\end{lemma}
\begin{proof}
For each $z\in U(1)$ and $x\in\F_n$, we have $\gamma_z(UexeU^*)=UexeU^*$. Hence 
$exeU^*\gamma_z(U) = U^*\gamma_z(U)exe$ and thus  $U^*\gamma_z(U)e$ belongs to 
$(e\F_n e)'\cap e\OO_n e$. Since this relative commutant is trivial, for each $z\in U(1)$ there exists a 
scalar $t(z)$ such that $\gamma_z(Ue) = t(z)Ue$. It follows that the mapping $t:U(1)\to\bC$ is a 
continuous character and consequently there exists an $m\in\bZ$ such that $t(z)=z^m$. 
We consider the following three cases, depending on the sign of $m$. 

\smallskip\noindent
(i) If $m=0$ then $Ue\in\F_n$ and hence $\tau(f)=\tau((Ue)(Ue)^*)=\tau((Ue)^*(Ue))=\tau(e)$. 

\smallskip\noindent
(ii) If $m>0$ then set $V:=UeS_1^{*m}$. Since $V$ belongs to $\F_n$, we have $\tau(f)=\tau(VV^*)=
\tau(V^*V)=\tau(S_1^m eS_1^{*m})=\tau(\varphi^m(e)S_1^mS_1^{*m})=\tau(e)/n^m$. 

\smallskip\noindent
(ii) If $m<0$ then set $V:=S_1^{-m}Ue$. Again $V\in\F_n$ and thus $\tau(e)=\tau(V^*V)=\tau(VV^*)=
\tau(S_1^{-m} fS_1^{*{-m}})=\tau(f)/n^{-m}$. 
\end{proof}

The following lemma is quite obvious but we give details since it allows us to reduce investigations of 
subalgebras of the general form (\ref{formofA}) to some special cases with conveniently chosen projections $e_j$.  

\begin{lemma}\label{lem2}
Let $e_1,\ldots,e_k$ and $f_1,\ldots,f_k$ be projections in $\F_n$ such that $\sum_{j=1}^k e_j = 1 = 
\sum_{j=1}^k f_j$ and $\tau(e_j)=\tau(f_j)$ for all $j$. Let $A=\bigoplus_{j=1}^k e_j\F_n e_j$ and 
$B=\bigoplus_{j=1}^k f_j\F_n f_j$. Then there exists a $u\in\U(\F_n)$ such that $uAu^*=B$. Hence we have 
$\N_{\OO_n}(A)\cong\N_{\OO_n}(B)$. 
\end{lemma}
\begin{proof}
For each $j=1,\ldots,k$ there exists a partial isometry $v_j\in\F_n$ such that $v_j^*v_j=e_j$ and 
$v_jv_j^*=f_j$. Set $u:=\sum_{j=1}^k v_j$. Then $ue_j\F_ne_ju^*=f_j\F_nf_j$ for each $j$, 
and the conclusion follows. 
\end{proof}

In view of Lemma \ref{lem2}, it suffices to consider those subalgebras $A$ of the form (\ref{formofA}) that 
all projections $e_j$ belong to the diagonal MASA $\D_n$. Each projection in $\D_n$ is a finite sum 
of projections $P_\mu$ for some words $\mu$. Before treating the general case, we note the following 
slight generalization of \cite[Example 1.18]{H}. 

\begin{example}\label{ex1}\rm 
Let $\mu_1,\ldots,\mu_k$ be words such that $\sum_{j=1}^k P_{\mu_j} =1$. Put 
$A=\bigoplus_{j=1}^k P_{\mu_j}\F_n P_{\mu_j}$. Then there is a natural isomorphism 
$$ \N_{\OO_n}(A) \cong \U(A) \rtimes {\mathcal S}_k, $$
where ${\mathcal S}_k$ is the symmetric group on $k$ letters. Indeed, for each permutation 
$\sigma\in\SS_k$ set $U_\sigma:=\sum_{j=1}^k S_{\mu_{\sigma(j)}}S_{\mu_j}^*$. It is easy to see 
that each $U_\sigma$ is unitary normalizing $A$, and that they form a group acting on $A$ by 
permuting the direct summands $ P_{\mu_j}\F_n P_{\mu_j}$. Thus we have an inclusion 
$\U(A) \rtimes {\mathcal S}_k \subseteq \N_{\OO_n}(A)$. For the reverse inclusion, take a 
$V\in\N_{\OO_n}(A)$. Considering $\Ad V$ action, we see that there exists 
a $\sigma\in\SS_k$ such that $VU_\sigma^*$ acts trivially on the center of $A$. Since 
$\N_{P_{\mu_j}\OO_n P_{\mu_j}}(P_{\mu_j}\F_n P_{\mu_j})=\U(P_{\mu_j}\F_n P_{\mu_j})$, for each $j$ we 
have $VU_\sigma^*P_{\mu_j}\in\U(P_{\mu_j}\F_n P_{\mu_j})$, and the claim easily follows. 
\hfill$\Box$
\end{example}

Now, we consider the general case of a $C^*$-subalgebra $A$ of the form (\ref{formofA}). Define 
an equivalence relation $\sim$ on the set $\{e_1,\ldots,e_k\}$ by 
\begin{equation}\label{equivalence}
e_i\sim e_j \Leftrightarrow \frac{\tau(e_i)}{\tau(e_j)}\in n^\bZ. 
\end{equation}
We denote by $\SS_\sim$ 
the subgroup of the permutation group of $\{e_1,\ldots,e_k\}$  consisting of those permutations which 
leave each of the equivalence classes of $\sim$ globally invariant. 

After this preparation, we are ready to prove our main result. 

\begin{theorem}\label{main}
Let $e_1,\ldots,e_k$ be non-zero projections in $\F_n$ such that $\sum_{j=1}^k e_j=1$, and let 
$A=\bigoplus_{j=1}^k e_j\F_n e_j$. Let $\SS_\sim$ be the corresponding subgroup of the permutation 
group of $\{e_1,\ldots,e_k\}$. Then there exists a natural group isomorphism 
$$ \N_{\OO_n}(A) \cong \U(A) \rtimes {\mathcal S}_\sim. $$
\end{theorem}
\begin{proof}
By Lemma \ref{lem2}, we may assume that each projection $e_j$ belongs to the diagonal MASA $\D_n$. 
Thus, there exist words $\mu_1,\mu_2,\ldots,\mu_N$, all of the same length and such that $\sum_{j=1}
^N P_{\mu_j}=1$, and there exist positive integers $m_1,m_2,\ldots,m_k$ such that $\sum_{j=1}^k m_j=N$ 
and $e_j=\sum_{i=m_{j-1}+1}^{m_j}P_{\mu_i}$ (here we put $m_0=0$) for each $j$. Relabelling, if 
necessary, we may assume that $\mu_{i_1}\prec\mu_{i_2}$ whenever $m_{j-1}+1 \leq i_1 \leq i_2 \leq m_j$. 

Now, let $\sigma\in\SS_\sim$. Take a $j\in\{1,\ldots,k\}$ and let $\sigma(e_j)=e_h$. There is an $m\in\bZ$ 
such that $m_j - m_{j-1} = (m_h-m_{h-1})n^m$. Suppose $m\geq 0$ (the case $m\leq0$ being treated 
analogously). We note that $e_h=\sum_{i=m_{h-1}+1}^{m_h}
\sum_{|\nu|=m}P_{\mu_i\nu}$. There is a unique $\prec$ order-preserving bijection 
\begin{equation}\label{bijectionpsi}
\psi:\{m_{h-1}+1,\ldots,m_h\} \times \{\nu : |\nu|=m\} \to \{m_{j-1}+1,\ldots,m_j\}, 
\end{equation}
that is,  
\begin{equation}\label{orderpsi}
\mu_{i_1}\nu_1 \prec \mu_{i_2}\nu_2 \; \Rightarrow \; \mu_{\psi(i_1,\nu_1)} \prec \mu_{\psi(i_2,\nu_2)}. 
\end{equation}
We set 
\begin{equation}\label{uj}
u_j:=\sum_{i=m_{h-1}+1}^{m_h}\sum_{|\nu|=m}S_{\mu_{\psi(i,\nu)}}S_{\mu_i\nu}^*. 
\end{equation}
By construction, we have $u_j^*u_j=e_h$ and $u_ju_j^*=e_j$. Observe that 
$$ S_{\mu_i\nu}S_{\mu_{\psi(i,\nu)}}^* \F_n S_{\mu_{\psi(i',\nu')}}S_{\mu_i'\nu'}^* \subseteq \F_n, $$
since $|\mu_i\nu|-|\mu_{\psi(i,\nu)}|+|\mu_{\psi(i',\nu')}|-|\mu_i'\nu'|=0$. It follows that 
\begin{equation}\label{ujaction}
u_j^*e_j\F_n e_j u_j = e_h\F_n e_h. 
\end{equation}
Now, we define 
\begin{equation}\label{usigma}
U_\sigma := \sum_{j=1}^k u_j^*. 
\end{equation}
Then $U_\sigma$ is an element of $\OO_n$ with the following properties:
\begin{description}
\item{[U1]} $U_\sigma$ is a unitary normalizing $A$. 
\item{[U2]} $U_\sigma e_j U_\sigma^*=\sigma(e_j)$ for each $j=1,\ldots,k$. 
\item{[U3]} For each $j=1,\ldots,k$ there exist words $\alpha_i,\beta_i$, $i=1,\ldots,r$ (for some $r\in\bN$) 
such that 
\begin{itemize}
\item $U_\sigma e_j = \sum_{i=1}^r S_{\alpha_i}S_{\beta_i}^*$, 
\item $\alpha_i \prec \alpha_{i'}$ whenever $\beta_i \prec \beta_{i'}$, and 
\item $|\alpha_i| - |\beta_i| = |\alpha_{i'}| - |\beta_{i'}|$ for all $i,i'=1,\ldots,r$. 
\end{itemize}
\end{description}
It is not difficult to verify that conditions $[U1]$--$[U3]$ characterize $U_\sigma$ uniquely. This uniqueness 
easily implies that $\{U_\sigma : \sigma\in\SS_\sim\}$ is a subgroup of $\U(\SS_\sim)$. Indeed, given 
$\sigma$ and $\sigma'$ in $\SS_\sim$, both $U_\sigma U_{\sigma'}$ and $U_{\sigma\sigma'}$ satisfy 
conditions $[U1]$--$[U3]$. Since the group  $\{U_\sigma : \sigma\in\SS_\sim\}$ is isomorphic to $\SS_\sim$ 
and acts on $\U(A)$ by $\Ad$, we have an inclusion $\U(A)\rtimes\SS_\sim \subseteq \N_{\OO_n}(A)$. 

To see that the reverse inclusion $\N_{\OO_n}(A) \subseteq \U(A)\rtimes\SS_\sim$ holds as well, take 
a $V$ in $\N_{\OO_n}(A)$. The action of $V$ on the center of $A$ yields a permutation $\sigma$ of 
$\{e_1,\ldots,e_k\}$. By Lemma \ref{lem1}, this permutation belongs to $\SS_\sim$. But then $VU_\sigma^*$ 
fixes each projection $e_j$, and thus it normalizes $e_j\F_n e_j$. Hence, for each $j$ there is a 
$w_j\in\U(e_j\F_n e_j)$ such that  $VU_\sigma e_j = w_j$. Putting $W:=\sum_{j=1}^k w_j$ we get a 
unitary in $A$ such that $V=WU_\sigma$, and the proof is complete. 
\end{proof}


\medskip\noindent
Tomohiro Hayashi \\
Nagoya Institute of Technology \\ 
Gokiso-cho, Showa-ku, Nagoya, Aichi \\
466--8555, Japan \\
E-mail: hayashi.tomohiro@nitech.ac.jp \\

\smallskip\noindent
Wojciech Szyma{\'n}ski\\
Department of Mathematics and Computer Science \\
The University of Southern Denmark \\
Campusvej 55, DK--5230 Odense M, Denmark \\
E-mail: szymanski@imada.sdu.dk

\end{document}